\documentclass[reqno, oneside, 12pt]{amsart}

\usepackage[letterpaper]{geometry}
\usepackage[cmyk,dvipsnames]{xcolor}

\usepackage{enumerate, hyperref,url,amssymb,amsmath,amsthm,amsxtra,mathtools,mathrsfs,calc,nccmath,xcolor, todonotes}
\usepackage{graphicx}
\usepackage{caption}
\usepackage{subcaption}

\usepackage{needspace}

\usepackage{mleftright}
\mleftright

\usepackage{xcolor}

\usepackage{enumitem}

\usepackage{nccmath}
\usepackage{cases}
\usepackage{comment}

\usepackage{hyperref}
\usepackage[final]{microtype}

\def\Z{\mathbb{Z}}

\usepackage{chngcntr}
\counterwithin{figure}{section}

\def\SL{{\rm SL}}

\newcommand{\texpdf}[2]{\texorpdfstring{#1}{#2}}

\renewcommand{\tilde}{\widetilde}

\newtheorem{theorem}{Theorem}[section]

\newtheorem{corollary}[theorem]{Corollary}
\newtheorem{proposition}[theorem]{Proposition}

\newtheorem*{theorem*}{Theorem}

\theoremstyle{remark}
\newtheorem*{remark}{Remark}
\newtheorem*{remarks}{Remarks}

\numberwithin{equation}{section}

\mathtoolsset{showonlyrefs}

\title{Eisenstein series modulo prime powers}

\date{\today}
\author{Scott Ahlgren}
 \address{Department of Mathematics\\
University of Illinois\\
Urbana, IL 61801}
 \email{sahlgren@illinois.edu}

\author{Cruz Castillo}
 \address{Department of Mathematics\\
University of Illinois\\
Urbana, IL 61801}
 \email{ccasti30@illinois.edu}

\author{Clayton Williams}
\address{Department of Mathematics\\
University of Illinois\\
Urbana, IL 61801} 
\email{cw78@illinois.edu} 

\thanks{
Scott Ahlgren
  was partially supported by  grant \#963004 from the Simons Foundation.
  Cruz Castillo was partially supported by the Alfred P. Sloan Foundation’s MPHD Program and by the National Science Foundation Graduate Research Fellowship grant DGE 21-46756.}

\begin{document}

\begin{abstract} If $p\geq 5$ is prime and $k\geq 4$ is an even integer with $(p-1)\nmid k$ we consider the  Eisenstein series $G_k$ on $\SL_2(\Z)$ modulo powers of $p$.
It is classically known that for such $k$ we have $G_k\equiv G_{k'}\pmod p$
if $k\equiv k'\pmod{p-1}$.
Here we obtain a generalization modulo prime powers  $p^m$ by giving an expression for $G_k\pmod{p^m}$ in terms of modular forms of weight at most $mp$.  As an application we extend a recent result of the first author with Hanson, Raum, and Richter by showing that, modulo powers of  $E_{p-1}$, every such Eisenstein series is congruent modulo $p^m$ to a modular form of weight at most $mp$.  We prove a similar result for the normalized Eisenstein series $E_k$ in the case that $(p-1)\mid k$ and $m<p$.

\end{abstract}

\maketitle

\section{Introduction}
For even integers $k\geq 2$, let 
 $B_k$ be the  Bernoulli number and define the weight $k$   Eisenstein series $G_k$ and $E_k$  by 
\begin{gather}
\label{eq:eis_series}
  G_k
:=
  -\frac{B_k}{2k}\,E_k
:=
  -\frac{B_k}{2k} + \sum_{n=1}^\infty \sigma_{k-1}(n)\,q^n,
\end{gather}
where $\sigma_{k-1}(n)$ is the sum of the $(k-1)$-st powers of the divisors of $n$.  For convenience we define $E_0:=1$.  Then $E_k$ is a modular form of weight $k$ on $\SL_2(\Z)$ unless $k=2$, in which case it is quasimodular.  The study of Eisenstein series modulo primes $p\geq 5$ has a long history; see, for example, 
\cite[\S 1]{Serre_padic}, \cite[\S3]{SwD_ladic}. 
We know for example that
\begin{gather}\label{eq:gk_integral}
     G_k \text{ is $p$-integral} \quad \text{if and only if}\quad  (p-1) \nmid k,
 \end{gather}
 and that 
 \begin{gather}
      E_k \equiv 1  \pmod p \qquad  \text{if } k  \equiv 0 \pmod{p-1}.
 \end{gather}
 From the Kummer congruences and properties of the sum-of-divisors function, we also know that 
\begin{gather}\label{eq:gk_modp}
    G_k \equiv G_{k'}  \pmod p \qquad \text{if } k \equiv k' \not\equiv 0 \pmod {p-1}.
\end{gather}
Some of these facts have straightforward generalizations to  prime power modulus; for example we 
have \cite[\S 1]{Serre_padic}
\begin{gather}
      E_k \equiv 1  \pmod {p^m} \qquad  \text{if } k  \equiv 0 \pmod{p^{m-1}(p-1)}.
 \end{gather}
It is also not difficult to show (see Section~\ref{sec:prelim})
that if $(p-1)\nmid k_0$ and $k_0>m$, then
\begin{gather}\label{eq:gk_mod_powersp}
    G_{k_0} \equiv G_{p^{m-1}(p-1)+k_0}  \pmod {p^m}.
\end{gather}

Throughout the paper we let $p\geq 5$ be a fixed prime, and we denote by $M_k$ the space of  modular forms of weight $k$
on $\SL_2(\Z)$ whose Fourier coefficients lie in the ring $\Z_{(p)}$ of $p$-integral rational numbers. We identify $f\in M_k$ with its Fourier expansion $\sum a(n)q^n\in \Z_{(p)}[\![q]\!]$,
and we interpret the congruence $\sum a(n)q^n\equiv \sum b(n)q^n\pmod{p^m}$ coefficient-wise.
The \emph{weight filtration} of a modular form $f$ modulo $p^m$ is defined as 
\begin{gather}\label{eq:weight_filtdef}
    \omega_{p^m}(f):=\inf\{k: \  f\equiv g\pmod{p^m}\quad \text{ for some}\quad  g\in M_k\}.
    \end{gather}
It follows from \eqref{eq:gk_mod_powersp} that every Eisenstein series $G_k$ with $k\geq 4$ and $(p-1)\nmid k$  has
\begin{gather}
    \label{eq:gk_filt_upper_bound}
    \omega_{p^m}\left(G_k\right)\leq m+p^{m-1}(p-1).
\end{gather}
Precise information about the properties of Eisenstein series modulo $p^2$ was obtained in   \cite[Theorem 1.1]{AHRR-Eis-2025}.  In particular, if $k\geq 4$ and $2\leq k_0\leq p-3$ has $k\equiv k_0\pmod{p-1}$, then it was shown that there exists $f_{(p-1)+k_0}\in M_{(p-1)+k_0}$ such that 
\begin{gather}\label{eq:AHRR_result}
    G_k\equiv  f_{(p-1)+k_0}E_{p-1}^n\pmod{p^2},
\end{gather}
where $n=(k-k_0)/(p-1)-1$ (this is trivially true when $4\leq k\leq 2p-4$).
This shows that (up to powers of $E_{p-1}$) every such Eisenstein series is determined mod $p^2$ by a modular form of weight at most $2p-4$.

The  goal of this paper is to obtain  analogues of \eqref{eq:gk_modp} and \eqref{eq:AHRR_result} modulo  arbitrary prime powers.
For example we will show that 
every Eisenstein series $G_k$ with $k\geq 4$ and  $(p-1)\nmid k$ is determined modulo $p^m$ (up to powers of $E_{p-1}$)  by a modular form of weight at most 
$mp$.
We also prove similar statements involving $E_k$ in the case when $(p-1)\mid k$.
To state the analogue of \eqref{eq:gk_modp} we define
\begin{gather}\label{eq:def_Hmkr}
H(m, \alpha, r):=(-1)^{m+1+r}\mbinom{\alpha-1-r}{m-1-r}\mbinom{\alpha}{r}, \qquad 0\leq r\leq m-1.
\end{gather}
\begin{theorem}\label{thm:main_G_k_cong}
Suppose that $p\geq 5$ is prime and  that $m\geq 1$.  
Let $k^*>m$ be an integer with $(p-1)\nmid k^*$.
Then for all $\alpha\geq 0$
 we have
\begin{gather}\label{eq:RHS_G_k_general}
    G_{\alpha(p-1)+k^*}
    \equiv \sum_{r=0}^{m-1}H(m, \alpha, r)G_{r(p-1)+k^*}E_{p-1}^{\alpha-r}\pmod{p^m}.
\end{gather}
\end{theorem}

\begin{remarks}
\begin{enumerate}
\item All terms in \eqref{eq:RHS_G_k_general} (and in \eqref{eq:ek_modp} below)
have the same weight.
    \item 
    Note that  
$H(m, \alpha, r)=
\delta_{r, \alpha}$ for $0\leq\alpha\leq m-1$ (where $\delta$ is the Kronecker delta symbol).
So the statement is trivially true for such $\alpha$.

\item Theorem~\ref{thm:main_G_k_cong} in the case  $m=1$ is equivalent to the congruence \eqref{eq:gk_modp}.
\item When $m=2$ and $k_0\geq 4$, the congruence \eqref{eq:AHRR_result} is implied by 
Theorem~\ref{thm:main_G_k_cong}.
This is not the case when $k_0=2$.
\item Given $k> m$ we can write $k=\alpha(p-1)+k^*$ with $m<k^*\leq m+p-1$ and $\alpha\geq 0$.  With these choices the  weights of the modular forms $G_{r(p-1)+k^*}$ appearing on the right side of $\eqref{eq:RHS_G_k_general}$ are at most $mp$.
\end{enumerate}
\end{remarks}

We obtain a similar result for $E_{k}$ in the case when $(p-1)\mid k$ and $m<p$.
\begin{theorem}\label{thm:main_E_k_cong}
Suppose that $p\geq 5$ is prime, that $1\leq m\leq p-1$, and that  $\alpha\geq 1$.
Then
\begin{gather}\label{eq:ek_modp}
    E_{\alpha(p-1)}\equiv \sum_{r=0}^{m-1}H(m, \alpha, r)E_{r(p-1)}E_{p-1}^{\alpha-r}\pmod{p^m}.
\end{gather}
\end{theorem}

The \emph{factor filtration} of a modular form modulo $p^m$ was introduced in \cite{ARR_theta}; this is a refinement of the weight  
filtration \eqref{eq:weight_filtdef} whose properties were crucial in determining large parts of the theta-cycle of  modular forms modulo $p^2$.
As an application of the results above we  give strong upper bounds for the  factor filtrations of Eisenstein series modulo any prime power.

For $m\geq 1$  let $\mathscr M_{m}\subseteq (\Z/p^m \Z)[\![q]\!]$ be the set of reductions of all elements of all  $M_k$.
We define the  modulo $p^m$ \emph{factor filtration} of $\overline f\in\mathscr{M}_{m}$ by 
    \begin{gather*}
    \tilde\omega_{p^m}(\overline f):=\inf\{k:\overline f\equiv g E_{p-1}^n\pmod{p^m}\text{ for some }n\geq 0\text{ and some }g\in M_k\}.
    \end{gather*}
By a slight  abuse of notation we write $\tilde{\omega}_{p^m}(f) = \tilde\omega_{p^m}(\overline f)$ 
when $f\in \Z_{(p)}[\![q]\!]$ has $\overline f\in \mathcal M_m$.
 
We will use the following notation: given $m\geq 1$ and a weight $k\geq 4$ we  define 
\begin{gather}\label{eq:def_k0}
\begin{aligned}
    k_0:=& \  \text{the least non-negative residue of $k\pmod{p-1}$, }\\
    k_0(m):=& \  \text{the smallest integer greater than $m$ and congruent to $k\!\!\pmod{p-1}$.}
\end{aligned}
\end{gather}
Then \eqref{eq:AHRR_result} is equivalent to the statement that for 
 $k\geq 4$ and $(p-1)\nmid k$
 we have
\begin{gather}\label{eq:AHRR_statement}
    \tilde\omega_{p^2}\left(G_k\right)\leq (p-1)+k_0.
\end{gather}

As a corollary of Theorem~\ref{thm:main_G_k_cong} we obtain an  analogous result modulo prime powers.   
\begin{corollary}
    \label{cor:main_factor_filt}
    Let $p\geq 5$ be  prime, let $m\geq 1$, and  let $k\geq 4$ have $(p-1)\nmid k$.
    Then
\begin{gather*}
    \tilde \omega_{p^m}\left(G_{k}\right)\leq(m-1)(p-1)+k_0(m).
\end{gather*}
\end{corollary}
\begin{remarks}
 \begin{enumerate}
     \item 
    When $m=2$ and  $k_0\geq 4$ this result implies \eqref{eq:AHRR_statement} (it does not imply \eqref{eq:AHRR_statement} in the case $k_0=2$).
    \item We have $k_0(m)\leq m+p-1$, so in all cases we have $\tilde\omega_{p^m}\left(G_k\right)\leq mp$.
    \end{enumerate}   
\end{remarks}

The bound in Corollary~\ref{cor:main_factor_filt} is often sharp, as can be computed in Mathematica \cite{wolfram}. For one example, let $p=7$, $m=8$, and $k = 337(p-1)+4=2026$. Then $k_0(m) = 10$ and $(m-1)(p-1)+k_0(m)=52$. Letting $\Delta$ denote the normalized cusp form of weight $12$, a
computation shows that 
\begin{gather*}
    G_k\equiv  f_1 E_{6}^{329}\pmod{7^8},
\end{gather*}
where 
\begin{multline*}
    f_1=  289118 E_4^{13} + 3330770 E_4^{10}\Delta +
 1615995 E_4^7 \Delta^2 + 4467661 E_4^4 \Delta^3 + 
 1172952  E_4\Delta^4\in M_{52}.
\end{multline*}
However, we find that there is no modular form $f_1'\in M_{46}$
with $f_1\equiv f_1'E_{6}\pmod{7^8}$. So the result is sharp in this case.

On the other hand, for  particular values of $m$ it is possible to give a  precise version of Corollary~\ref{cor:main_factor_filt} with improved bounds in many cases (although the complexity of the statement increases quickly with $m$).  We will give a complete treatment of the cases $m=3$  and $m=4$ in Section~\ref{sec_m=3_4}.
For example, we will show that 
if $k_0\geq 4$ then we have
\begin{align*}
    \tilde \omega_{p^3}\left(G_{\alpha(p-1)+k_0}\right)\leq\begin{cases}
        (p-1)+k_0, &\text{if} \ 
        \alpha\equiv 0,1\pmod{p};\\
        2(p-1)+k_0, &\text{otherwise}.
    \end{cases}
\end{align*}

We also consider  the case when $k\equiv 0\pmod{p-1}$.
Here computations suggest that the analogue of Corollary~\ref{cor:main_factor_filt} is true; in other words if $(p-1)\mid k$ (i.e., $k_0=0$) then we have 
\begin{gather}\label{eq:general_powers_speculation}
    \tilde \omega_{p^m}\left(E_{k}\right)\leq(m-1)(p-1)+k_0(m).
\end{gather}
This statement would follow from  an unproved congruence involving Bernoulli numbers which is discussed in Section~\ref{sec:speculation}.  
As a corollary to Theorem~\ref{thm:main_E_k_cong} we obtain a stronger result for small~$m$. 
\begin{corollary}
    \label{cor:E_beta (p-1)_fact_filt}
Suppose that  $k\in\Z_{\geq 0}$ has $k\equiv 0\pmod{p-1}$ and that $1\leq m\leq p-1$. Then
    \begin{gather}
\tilde\omega_{p^m}\left(E_k\right) \leq (m-1)(p-1).
    \end{gather}
\end{corollary}
\begin{remark}
    When $m<p$ and $k_0=0$  we have $k_0(m)=p-1$, so the bound in Corollary~\ref{cor:E_beta (p-1)_fact_filt} is stronger than \eqref{eq:general_powers_speculation} in this case.
\end{remark}
This result is also  sharp in general. For an example, let $p=17$,  $k=81(p-1)=1296$, and $m=6$.
A computation shows that 
\begin{gather}
    E_k\equiv f_2E_{16}^{76}\pmod{17^6},
\end{gather}
where
\begin{multline}
    f_2=E_4^{20}+17835578 E_4^{17}\Delta+1427399 E_4^{14}\Delta^2+23585491
E_4^{11}\Delta^3+19629555E_4^{8}\Delta^4\\
+23614096
E_4^{5}\Delta^5+44217E_4^{2}\Delta^6\in M_{80}.
\end{multline}
It can be checked that there is no $f_2'\in M_{64}$ with $f_2\equiv f_2'E_{16}\pmod {17^6}$.

To prove the results in the case $(p-1)\nmid k$ we begin with a congruence involving Bernoulli numbers due to Sun \cite{Sun-Bernoulli} which implies that the constant terms in \eqref{eq:RHS_G_k_general} agree modulo $p^m$.  In Section~\ref{sec:proof_thm_1.1} we show that this  extends first to a congruence involving Eisenstein series of different weights and finally to the statement of Theorem~\ref{thm:main_G_k_cong}.  To prove this we use a multi-parameter combinatorial identity  which is proved in Proposition~\ref{prop:combin_H}.
In Section~\ref{sec:proof_thm_1.2} we begin by proving a crucial Bernoulli number congruence (Proposition~\ref{prop:E_p-1_Bernoulli_cong}) and then use arguments as in  Section~\ref{sec:proof_thm_1.1} to prove Theorem~\ref{thm:main_E_k_cong}.   
In Section~\ref{sec_m=3_4} we give precise statements in the case when $m=3$ or $4$, and in the last section we discuss an analogue of Theorem~\ref{thm:main_E_k_cong} for arbitrary $m$. 

\subsection*{Acknowledgments}We thank Carsten Schneider for   helpful advice
 regarding the use of his software package Sigma in the proof of Proposition~\ref{prop:combin_H}.
We are also grateful to the referees for their helpful comments.

\section{Preliminaries}\label{sec:prelim}
We recall some  facts about Bernoulli numbers which can be found for example    in \cite[\S 9.5]{CohNT07}. Let $p\geq 5$ be prime, let $k,k',$ and $r$ be positive integers with $k,k'$ even, and let $\nu_{p}$ denote the $p$-adic valuation. The  Clausen-von Staudt theorem states that 
\begin{gather}\label{eq:CVSthm}
    B_k \equiv - \sum_{\substack{q\text{ prime}\\(q-1)|k}} \frac 1 q\pmod 1,
\end{gather}
which gives 
\begin{gather}\label{eq:Bvalp-1}
      \nu_p\left(\frac{B_k}{k}\right) = -\nu_p(k)-1 \quad \text{and} \quad pB_k \equiv -1 \pmod{p} \quad  \text{if }(p-1)|k.
\end{gather}
On the other hand, we have  
\begin{gather}\label{eq:Bvalnotp-1}
    \nu_p\left(\frac{B_k}{k}\right) \geq 0 \quad \text{for } (p-1)\nmid k
\end{gather}
(note that  \eqref{eq:gk_integral} follows from these facts).
 The Kummer congruences imply that if $(p-1)\nmid k$ and $k\equiv k' \pmod{p^{r-1}(p-1)},$ then 
\begin{gather}\label{eq:kummer}
    (1-p^{k-1})\frac{B_k}{k} \equiv (1-p^{k'-1})\frac{B_{k'}}{k'} \pmod{p^r}.
\end{gather}

These congruences imply the claim   \eqref{eq:gk_mod_powersp};
  when $k= k_0 +{p^{m-1}(p-1)}$ and $k_0> m$, it follows from \eqref{eq:kummer} that the 
  constant terms of $G_{k_0}$ and $G_{k}$ are congruent modulo $p^m$.  By Euler's theorem we have $\sigma_{k_0-1}(n) \equiv \sigma_{k-1}(n) \pmod{p^m}$, which shows that  the non-constant terms are also congruent.

In the papers \cite{Sun97,Sun-Bernoulli}, Sun proved a number of congruences for Bernoulli polynomials modulo prime powers.  Recall the definition \eqref{eq:def_Hmkr} of $H(m, \alpha, r)$. 
By \cite[Lemma 2.1]{Sun97} we have the following for any function $f$:
\begin{gather}\label{eq:H_inversion}
f(\alpha)=\sum_{r=0}^{n-1}H(n, \alpha, r)f(r)+
\sum_{r=n}^\alpha\mbinom{\alpha}{r}(-1)^r\sum_{s=0}^r\mbinom{r}{s}(-1)^sf(s).    
\end{gather}
Let $p$ be a prime and $f: \Z_{\geq0}\to \Z_{(p)}$ be a function.  Following \cite{Sun-Bernoulli}, we call $f$ \emph{$p$-regular} if 
\begin{gather}
    \label{eq:def_p_regular}
    \sum_{k=0}^n\mbinom{n}{k}(-1)^k f(k)\equiv 0\pmod{p^n}\quad \text{for all}\quad n\in \Z_{>0}.
\end{gather}
We will need the following facts from \cite[\S 2]{Sun-Bernoulli}:
\begin{proposition}
    \label{prop:sun_regular}
Let $p$ be a prime.
    \begin{enumerate}
        \item The product of $p$-regular functions is $p$-regular.
        \item If $f$ is $p$-regular then for all $\alpha\geq 1$ and $m\geq 1$ we have 
        \begin{gather}
            f(\alpha)=\sum_{r=0}^{m-1}H(m, \alpha, r)f(r)\pmod{p^m}. 
        \end{gather}
    \end{enumerate}
\end{proposition}

\section{Proof of Theorem~\ref{thm:main_G_k_cong} and Corollary~\ref{cor:main_factor_filt}}
\label{sec:proof_thm_1.1}
We begin by proving a congruence involving modular forms of different weights.
\begin{proposition}
\label{prop:main_G_k_cong}
    Suppose that $p\geq 5$ is prime and  that $m\geq 1$.  
Let $k^*>m$ be an integer with $(p-1)\nmid k^*$.
Then for all $\alpha\geq 0$
 we have
\begin{gather*}
    G_{\alpha(p-1)+k^*}
    \equiv \sum_{r=0}^{m-1}H(m, \alpha, r)G_{r(p-1)+k^*
    }\pmod{p^m}.
\end{gather*}

\end{proposition}

\begin{proof}[Proof of Proposition~\ref{prop:main_G_k_cong}]
Since $k^*>m$, the congruence of the constant terms  follows from \cite[Corollary~4.1]
{Sun-Bernoulli}.  
 To prove that the non-constant terms agree, it is enough to show that 
\begin{gather*}
    \sigma_{\alpha(p-1)+k^*-1}(n)
    \equiv \sum_{r=0}^{m-1}H(m, \alpha, r)\sigma_{r(p-1)+k^*-1}(n)\pmod{p^m}\quad \text{for all $n\geq 1$}.
\end{gather*}
Since $k^*>m$ it is enough to prove that for $p\nmid d$ we have
\begin{gather}\label{eq:divisor_show}
    d^{\alpha(p-1)}
    \equiv \sum_{r=0}^{m-1}H(m, \alpha, r)d^{r(p-1)}\pmod{p^m}.
\end{gather}
Since
\begin{gather*}
    (1-d^{p-1})^n=\sum_{k=0}^n\mbinom{n}{k}(-1)^k d^{k(p-1)},
\end{gather*}
we see that the function $k
\mapsto d^{k(p-1)}$ is  
 $p$-regular if $p\nmid d$.
 Then  \eqref{eq:divisor_show} follows from Proposition~\ref{prop:sun_regular},  
and the proposition is proved.
\end{proof}

\begin{proof}[Proof of Theorem~\ref{thm:main_G_k_cong}]
Write $E_{p-1}=1+pE$ and expand
\begin{gather*}
    E_{p-1}^{\alpha-r}\equiv \sum _{j=0}^{m-1}\mbinom{\alpha-r}{j}p^j E^j\pmod{p^m}.
\end{gather*}
The right side of 
\eqref{eq:RHS_G_k_general} becomes
\begin{gather}\label{eq:RHS_G_k_small_m_expand}
   \sum _{j=0}^{m-1}p^j E^j \sum_{r=0}^{m-1}\mbinom{\alpha-r}{j}H(m, \alpha, r)G_{r(p-1)+k^*}\pmod{p^m}.
\end{gather}
By Proposition \ref{prop:main_G_k_cong}, 
the $j=0$ term in \eqref{eq:RHS_G_k_small_m_expand} gives the left side of \eqref{eq:RHS_G_k_general} modulo $p^m$.

To treat the terms with $j\geq 1$ we
 expand each Eisenstein series $G_{r(p-1)+k^*}$ modulo $p^{m-j}$ using Proposition~\ref{prop:main_G_k_cong} and rearrange to find that 
\begin{gather}
\begin{aligned}\label{eq:RHS_G_k_small_minusj}
  \sum_{r=0}^{m-1}\mbinom{\alpha-r}{j}&H(m, \alpha, r)G_{r(p-1)+k^*}\\
  &\equiv   
  \sum_{r=0}^{m-1}\mbinom{\alpha-r}{j}H(m, \alpha, r)\sum_{s=0}^{m-j-1}H(m-j, r,s) G_{s(p-1)+k^*}\\
  &\equiv \sum_{s=0}^{m-j-1}G_{s(p-1)+k^*}\sum_{r=0}^{m-1}\mbinom{\alpha-r}{j}H(m, \alpha, r)H(m-j, r,s)\pmod{p^{m-j}}.
\end{aligned}
\end{gather}
Theorem~\ref{thm:main_G_k_cong}  follows from 
\eqref{eq:RHS_G_k_small_m_expand},
\eqref{eq:RHS_G_k_small_minusj},  and the next proposition (recall from the definition \eqref{eq:def_Hmkr} that $H(m-j, r, s)=0$  for $r<s$).
\end{proof}
\begin{proposition}
    \label{prop:combin_H}
    For $1\leq j\leq m-1$, $0\leq s\leq m-j-1$, and $\alpha\geq 0$ we have 
    \begin{gather}\label{eq:Sumdef}
      \sum_{r=s}^{m-1}\mbinom{\alpha-r}{j} H(m, \alpha, r)H(m-j, r,s)=0.  
    \end{gather}
\end{proposition}
\begin{proof}
To analyze this sum we use the Mathematica package Sigma developed by Carsten Schneider \cite{Schneider_sigma} (we are grateful to him for advice regarding its use).
Let $F(m,r)$ be the summand in \eqref{eq:Sumdef}; we have
\begin{gather*}
    F(m, r)=(-1)^{r+j+s}\mbinom{\alpha -r}{j}\mbinom{\alpha-1-r}{m-1-r}\mbinom{\alpha }{r}\mbinom{r-1-s}{m-j-1-s} \mbinom{r}{s}.    
\end{gather*}
The creative telescoping algorithm in Sigma produces the function
\begin{gather*}
       G(r):=(-1)^{r+j+s}
       \frac{(s-r) (j+r-\alpha)
   \binom{r}{s} \binom{\alpha}{r}\binom{\alpha -r}{j} \binom{\alpha-1-r}{m-1-r} 
   \binom{r-1-s}{m-j-1-s}}{m-j-s}
   \end{gather*}
with the following property: 
\begin{gather}\label{eq:telescope}
(\alpha-m)F(m,r)+(m-s)F(m+1, r)=G(r)-G(r-1).
\end{gather}
   Note that $G(r)$ is defined for all values of the parameters in the proposition since $m-j-s>0$.
Details on how Sigma produces the function $G(r)$ are given in \cite{Schneider_sigma}.  The important fact for our purposes is that equation \eqref{eq:telescope}, once it is known, can be verified by a routine computation.  Indeed, both sides of the equation reduce to 
\begin{gather*}
    \mfrac{(-1)^{j+r+s}\Gamma (\alpha +1) \Gamma (\alpha -r) }{\Gamma (j)
   \Gamma (s+1) \Gamma (\alpha -m) \Gamma (m-r+1) \Gamma (\alpha -j-r+1)
   \Gamma (j-m+r+1) \Gamma (-j+m-s+1)}.
\end{gather*}

Let  $S(m)$ be the sum in \eqref{eq:Sumdef}.  Summing  \eqref{eq:telescope} from $r=s$ to $m-1$ gives 
\begin{gather}\label{eq:telescope1}
 (\alpha-m)S(m)+(m-s)S(m+1)=(m-s)F(m+1,m)+G(m-1)-G(s-1).   
\end{gather}
It is clear from the definition
that $G(s-1)=0$, and   
a computation shows that
\begin{gather*}
    -G(m-1)=\mfrac{(-1)^{j+m+s}\Gamma (\alpha +1) }{\Gamma (j) \Gamma (j+1) \Gamma
   (s+1) \Gamma (\alpha -j-m+1) \Gamma (-j+m-s+1)}=(m-s)F(m+1,m).
\end{gather*}
It follows from \eqref{eq:telescope1}
that
\begin{gather}\label{eq:Sum_telescope}
 (\alpha-m)S(m)+(m-s)S(m+1)=0.   
\end{gather}

To finish, fix $j\geq 1$ and $s\geq 0$.    We must  prove that $S(m)=0$ for all $m\geq s+j+1$; from the recurrence \eqref{eq:Sum_telescope} it will suffice to prove that $S(s+j+1)=0$.
To this end we compute
\begin{gather}
    S(s+j+1)=\sum_{r=s}^{s+j}
    (-1)^{r+j+s}
    \mbinom{\alpha-r}{j}
    \mbinom{\alpha-1-r}{s+j-r}
    \mbinom{\alpha}{r}
    \mbinom{r}{s}.
\end{gather}
If $\alpha\leq s+j$ then the second binomial coefficient is  zero and we are done.

When  $\alpha>s+j$ we simplify as follows with $\beta=\alpha-s>j$:
\begin{align}
    \label{eq:creative_two}
    S(s+j+1)
    &=\sum_{r=0}^{j}
    (-1)^{r+j}
    \mbinom{\alpha-r-s}{j}
    \mbinom{\alpha-1-r-s}{j-r}
    \mbinom{\alpha}{r+s}
    \mbinom{r+s}{s}\\
    &=(-1)^j\mbinom{\alpha}{s}\sum_{r=0}^{j}
    (-1)^{r}
    \mbinom{\alpha-r-s}{j}
    \mbinom{\alpha-1-r-s}{j-r}
    \mbinom{\alpha-s}{r}\\
&=(-1)^j\mbinom{\beta+s}{s}\sum_{r=0}^{j}
    (-1)^{r}
    \mbinom{\beta-r}{j}
    \mbinom{\beta-1-r}{j-r}
    \mbinom{\beta}{r}.
\end{align}
A short computation shows that
\begin{gather}
   S(s+j+1) =(-1)^j\mbinom{\beta+s}{s}\mbinom{\beta}{j}
   \mbinom{\beta -1}{j} \,
   _2F_1(-j,j-\beta ;1-\beta
   ;1).
\end{gather}
By the  Chu-Vandermonde theorem \cite[Corollary 2.2.3]{AndrewAskeyRoy}, the hypergeometric function evaluates to 
\begin{gather}
    \frac{(1-j)_j}{(1-\beta)_j},
\end{gather}
where $(a)_j=a(a+1)\dots(a+j-1)$ is the Pochammer symbol.
This finishes the proof since 
the denominator is non-zero when  $\beta>j$.
\end{proof}

\begin{proof}[Proof of Corollary~\ref{cor:main_factor_filt}]
We may  assume that  $k > (m-1)(p-1) + k_0(m)$;
otherwise the result clearly holds.   Writing 
 $k = \alpha(p-1)+k_0(m)$
 with $\alpha>m-1$, Theorem~\ref{thm:main_G_k_cong} shows that there exists 
 $g\in M_{(m-1)(p-1)+k_0(m)}$ with 
\begin{gather}
    G_{\alpha(p-1)+k_0(m)} \equiv g E_{p-1}^{\alpha -m+1} \pmod{p^m},
\end{gather}
 which  establishes Corollary~\ref{cor:main_factor_filt}.
\end{proof}

\section{Proof of Theorem~\ref{thm:main_E_k_cong} and Corollary~\ref{cor:E_beta (p-1)_fact_filt}}
\label{sec:proof_thm_1.2}

To treat weights which are divisible by $p-1$ we begin by proving the following congruence for Bernoulli numbers. 
\begin{proposition}\label{prop:E_p-1_Bernoulli_cong}
Suppose that $p\geq 5$ is prime, that  $\alpha\geq 1$,  and that $1\leq m\leq p-1$.
Then for any positive integer $d$ with $p\nmid d$ we have 
\begin{gather*}
d^{\alpha(p-1)}\,\frac{\alpha }{B_{\alpha(p-1)}}\equiv \sum_{r=1}^{m-1}H(m, \alpha, r)\,d^{r(p-1)}\,\frac {r}{B_{r(p-1)}}\pmod{p^m}.
\end{gather*}
\end{proposition}
\begin{proof}[Proof of Proposition~\ref{prop:E_p-1_Bernoulli_cong}]
Define the function 
\begin{gather}
    \label{eq:def_fk}
    f(k):=\left(p-p^{k(p-1)}\right)B_{k(p-1)}\quad\text{for}\quad k\geq 0.
\end{gather}
If $n\geq 1$ then by \cite[Theorem~3.1]{Sun97} we have
\begin{gather}
    \label{eq:sun97cong}
    \sum_{k=0}^n\mbinom{n}{k}(-1)^k f(k)\equiv
    \begin{cases}
        0\ \ \ \ \pmod{p^n},\quad&\text{if}\quad (p-1)\nmid n;\\
        p^{n-1}\pmod{p^n},\quad&\text{if}\quad (p-1)\mid n.
    \end{cases}
\end{gather}
Define the sequence $\{a(n)\}$ by 
\begin{gather}
    \label{eq:def_an}
    a(n):=\begin{cases}
        0,\quad&\text{if}\quad n=0\quad\ \text{or}\quad\  (p-1)\nmid n;\\
        -p^{n-1},\quad&\text{if}\quad n>0\quad\text{and}\quad (p-1)\mid n,
    \end{cases}
\end{gather}
and the function $g(k)$ by 
\begin{gather}
    \label{eq:def_gk}
    g(k):=\sum_{n=0}^k \mbinom{k}{n}(-1)^n a(n)\quad\text{for}\quad k\geq 0.
\end{gather}
From  binomial inversion we have
\begin{gather}
\sum_{k=0}^n \mbinom{n}{k}(-1)^k g(k)=a(n);
\end{gather}
it follows from \eqref{eq:sun97cong} that the function  $f(k)+g(k)$ is 
$p$-regular. 

Now let $n\in\Z_{>0}$. 
By \eqref{eq:Bvalp-1} we have $p\nmid(f(k)+g(k))$.
It follows from Proposition~\ref{prop:sun_regular} that 
$\left(f(k)+g(k)\right)^{\phi(p^n)-1}$ is $p$-regular.
Since
\begin{gather*}
    \sum_{k=0}^n\mbinom{n}{k}(-1)^k\frac1{f(k)+g(k)}\equiv \sum_{k=0}^n\mbinom{n}{k}(-1)^k\left(f(k)+g(k)\right)^{\phi(p^n)-1}\equiv 0\pmod{p^n}
\end{gather*}
we conclude  that
$1/(f(k)+g(k))$ is also $p$-regular.
From the identity
\begin{gather}\label{eq:binom_kron_delta}
\sum_{k=0}^n\mbinom{n}{k}(-1)^k k=-\delta_{n 1}
\end{gather}
we see that the function $k\mapsto p k$ is  $p$-regular.
Recalling that the same is true of 
$k\mapsto d^{k(p-1)}$ when $p\nmid d$, 
we deduce from Proposition~\ref{prop:sun_regular} that for $\alpha, m\geq 1$ and $p\nmid d$ we have
\begin{gather}
    \label{eq:invert_f_plus_g}
   d^{\alpha(p-1)}\, \frac{p \alpha }{f(\alpha)+g(\alpha)}\equiv 
    \sum_{r=1}^{m-1}H(m, \alpha, r)d^{r(p-1)}\frac {pr}{f(r)+g(r)}\pmod{p^m}.
\end{gather}
From \eqref{eq:Bvalp-1} and \eqref{eq:def_fk} we see that $f(r)$ is a $p$-unit and that 
\begin{gather*}
    f(r)\equiv p B_{r(p-1)}\pmod{p^{p-2}}.
\end{gather*}
Furthermore   \eqref{eq:def_an} and \eqref{eq:def_gk} show that 
\begin{gather*}
    g(r)\equiv 0\pmod{p^{p-2}}.
\end{gather*}
Combining these facts gives
\begin{gather*}
\label{eq:bernoulli_last_cong}
    \frac{pr}{f(r)+g(r)}\equiv
    \frac {pr}{pB_{r(p-1)}}\equiv\frac {r}{B_{r(p-1)}}\pmod{p^{p-1}}
    \quad \text{for}\quad r\geq 1.
\end{gather*}
The proposition follows from this congruence together with 
\eqref{eq:invert_f_plus_g} 
since $p-1\geq m$.
\end{proof}

We  use Proposition~\ref{prop:E_p-1_Bernoulli_cong} to prove the analogous congruence between modular forms of varying weights. 
\begin{proposition}
\label{prop:main_E_k_cong}
Suppose that $p\geq 5$ is prime, that  $\alpha\geq 1$,  and that $1\leq m\leq p-1$.
Then
\begin{gather*}
    E_{\alpha(p-1)}\equiv \sum_{r=0}^{m-1}H(m, \alpha, r)E_{r(p-1)}\pmod{p^m}.
\end{gather*}
\end{proposition}
\begin{proof}
We prove this congruence term by term. To see that the constant terms on each side agree, we use 
 \eqref{eq:H_inversion} with $f(s) =1$ 
 and the fact that 
\begin{gather}
\sum_{k=0}^n\mbinom{n}{k}(-1)^k =\delta_{n 0}.
\end{gather}

By Proposition~\ref{prop:E_p-1_Bernoulli_cong}, when $p\nmid d$ we have 
 \begin{gather}\label{eq:p-1_div}
    d^{\alpha(p-1)-1}\frac{\alpha }{B_{\alpha(p-1)}} \equiv \sum_{r=1}^{m-1}H(m,\alpha,r) d^{r(p-1)-1}\frac{r}{B_{r(p-1)}} \pmod{p^m}.
\end{gather}
From the first assertion of \eqref{eq:Bvalp-1} we see that when $p\mid d$ we have 
\begin{gather}
    d^{r(p-1)-1}\frac{r }{B_{r(p-1)}}\equiv 0\pmod{p^{p-1}},\quad r\geq 1.
\end{gather}
Since $p-1\geq m$ it follows that  for every positive $n$ we have
\begin{gather}
    \frac{\alpha }{B_{\alpha(p-1)}}\sigma_{\alpha(p-1) -1}(n) \equiv \sum_{r=1}^{m-1}H(m,\alpha,r) \frac{r}{B_{r(p-1)}} \sigma_{r(p-1) -1}(n) \pmod{p^m},
\end{gather}
which shows that the non-constant terms agree and proves the proposition.
\end{proof}
\begin{proof}[Proof of Theorem~\ref{thm:main_E_k_cong}]
We proceed as in the proof of Theorem~\ref{thm:main_G_k_cong}; writing 
 $E_{p-1} = 1+pE$ the right side of \eqref{eq:ek_modp} becomes 
\begin{gather}\label{eq:RHS_E_k_small_m_expand}
    \sum_{j=0}^{m-1}p^jE^j\sum_{r=0}^{m-1}\mbinom{\alpha-r}{j}H(m, \alpha, r)E_{r(p-1)}\pmod{p^m}.
\end{gather}
The $j=0$ term gives 
the left side of \eqref{eq:ek_modp} by 
Proposition~\ref{prop:main_E_k_cong}.
To show that the other terms vanish modulo $p^m$ we 
 proceed as before.  In particular, 
 expanding each $E_{r(p-1)}$ modulo $p^{m-j}$ using Proposition~\ref{prop:main_E_k_cong} and rearranging leads
 again to the combinatorial identity of Proposition~\ref{prop:combin_H}.
\end{proof}
\begin{proof}[Proof of Corollary~\ref{cor:E_beta (p-1)_fact_filt}] 
This follows immediately from Theorem~\ref{thm:main_E_k_cong}.
\end{proof}

\section{Congruences modulo \texpdf{$p^3$}{p cubed} and \texpdf{$p^4$}{p fourth}}\label{sec_m=3_4}
Here we give more precise versions of Corollary~\ref{cor:main_factor_filt} when $m=3$ and $m=4$.  The statements rapidly become more complicated as $m$ increases.
\begin{corollary}\label{cor:main_factor_filt_p_cubed}
    Let $p\geq 5$ be  prime and write $k\geq 4$ as $k=\alpha(p-1)+k_0$ with $2\leq k_0\leq p-3.$ 
    \begin{enumerate}
        \item 
If $k_0\geq 4$ then 
\begin{align*}
    \tilde \omega_{p^3}\left(G_{k}\right)\leq\begin{cases}
        (p-1)+k_0, &\text{if} \ 
        \alpha\equiv 0,1\pmod{p};\\
        2(p-1)+k_0, &\text{otherwise}.
    \end{cases}
\end{align*}
\item 
If $k_0= 2$ then 
\begin{align*}
    \tilde \omega_{p^3}\left(G_{k}\right)\leq\begin{cases}
        (p-1)+2, &\text{if} \ 
        \alpha\equiv 1\pmod{p};\\
    2(p-1)+2, &\text{if} \ 
        \alpha\equiv 2\pmod{p};\\
        3(p-1)+2, &\text{otherwise}.
    \end{cases}
\end{align*}
\end{enumerate}
\end{corollary}
\begin{corollary}\label{cor:main_factor_filt_p_fourth}
    Let $p\geq 5$ be  prime and write $k\geq 4$ as $k=\alpha(p-1)+k_0$ with $2\leq k_0\leq p-3.$ 
    \begin{enumerate}
        \item 
If $k_0\geq 6$ then 
\begin{align*}
    \tilde \omega_{p^4}\left(G_{k}\right)\leq\begin{cases}
        (p-1)+k_0, &\text{if }  
        \alpha\equiv 0,1\pmod{p^2};\\
        2(p-1)+k_0, &\text{if }\alpha\equiv 0,1,2\pmod{p};\\
        3(p-1)+k_0,&\text{otherwise}.
    \end{cases}
\end{align*}
\item 
If $k_0= 4$ then
\begin{align*}
    \tilde \omega_{p^4}\left(G_{k}\right)\leq\begin{cases}
        (p-1)+4, &\text{if }  
        \alpha\equiv 1\pmod{p^2};\\
    2(p-1)+4, &\text{if } 
        \alpha\equiv 1,2\pmod{p};\\
        3(p-1)+4, &\text{if }\alpha\equiv 3\pmod{p};\\
        4(p-1)+4,&\text{otherwise}.
    \end{cases}
\end{align*}
\item
If $k_0= 2$ then
\begin{align*}
    \tilde \omega_{p^4}\left(G_{k}\right)\leq\begin{cases}
        (p-1)+2, &\text{if }  
        \alpha\equiv 1\pmod{p^2};\\
    2(p-1)+2, &\text{if }  
        \alpha\equiv 2\pmod{p^2};\\
        3(p-1)+2, &\text{if }\alpha\equiv 1,2,3\pmod{p};\\
        4(p-1)+2,&\text{otherwise}.
    \end{cases}
\end{align*}
\end{enumerate}

\end{corollary}
\begin{proof}[Proof of  Corollary~\ref{cor:main_factor_filt_p_cubed}.]
 The general cases 
\begin{gather}\tilde\omega_{p^3}\left(G_k\right)\leq
\begin{cases} 2(p-1)+k_0,&\text{if }k_0\geq 4;\\
    3(p-1)+2,&\text{if }k_0=2\end{cases}
    \end{gather}
    follow from Corollary \ref{cor:main_factor_filt} and the fact that $k_0(3)=k_0$ if $k_0\geq 4$ and $k_0(3)=p+1$ if $k_0=2$.

To prove the remaining statement 
when $k_0\geq 4$, we use  Theorem~\ref{thm:main_G_k_cong} to write
\begin{multline}\label{eq:main_cong_k0large_3}
        G_{\alpha(p-1)+k_0}\equiv \mbinom{\alpha-1}{2}G_{k_0}    E_{p-1}^\alpha
        -\alpha(\alpha-2)G_{(p-1)+k_0}E_{p-1}^{\alpha-1}\\
        +\mbinom{\alpha}{2}G_{2(p-1)+k_0}E_{p-1}^{\alpha-2}\pmod{p^3}.
        \end{multline}
It is clear from the definition that if $m\geq 1$
  and if $f$, $g$ are modular forms of weight $k$ modulo $p^m$ for some $k$, then 
\begin{gather}\label{eq:fact_filt_fact1}
    \tilde\omega_{p^{m+1}}(p f)=\tilde\omega_{p^m}(f)\qquad \text{and}\qquad
    \tilde \omega_{p^m}( f+ g)\leq \max\{\tilde\omega_{p^m}( f),\, \tilde\omega_{p^m}(g)\}.
  \end{gather}
When $\alpha\equiv 0,1\pmod{p}$ we have $\binom{\alpha}{2}\equiv 0\pmod{p}$.
Using this fact with 
\eqref{eq:main_cong_k0large_3}
and \eqref{eq:fact_filt_fact1}
gives
    \begin{gather*}
\tilde\omega_{p^3}\left(G_{k}\right)\leq \max\{(p-1)+k_0,\ \tilde\omega_{p^2}\left(G_{2(p-1)+k_0}\right)\},
    \end{gather*}
From Corollary~\ref{cor:main_factor_filt} in the case $m=2$ we conclude that $\tilde\omega_{p^3}\left(G_{k}\right)\leq(p-1)+k_0$, as desired.

 If $k_0=2$ then Theorem~\ref{thm:main_G_k_cong} with $k^*=p+1$ and $\alpha$ replaced by $\alpha-1$ gives
\begin{multline}\label{eq:main_con_k0=2_3}
        G_{\alpha(p-1)+2}\equiv \mbinom{\alpha-2}{2}G_{(p-1)+2}E_{p-1}^{\alpha-1}
        -(\alpha-1)(\alpha-3)G_{2(p-1)+2}E_{p-1}^{\alpha-2}\\+\mbinom{\alpha-1}{2}G_{3(p-1)+2}E_{p-1}^{\alpha-3}\pmod{p^3}.
        \end{multline}
  The claims when $\alpha\equiv 1, 2\pmod p$ follow from an  analysis as above.
\end{proof}

\begin{proof}[Proof of  Corollary~\ref{cor:main_factor_filt_p_fourth}.]
Since the proofs use similar methods   we  discuss only the case when $k_0\leq 4$ and $\alpha\equiv 1\pmod{p}$ for brevity.  Theorem~\ref{thm:main_G_k_cong} with $k^*=k_0+p-1$  and $\alpha$ replaced by $\alpha-1$ 
gives
\begin{multline}
        G_{\alpha(p - 1)+k_0}\equiv 
- \mbinom{\alpha - 2}{3} G_{(p - 1)+k_0}E_{p-1}^{\alpha-1}
+ (\alpha - 1)\mbinom{\alpha - 3}{2}  G_{2(p - 1)+k_0} E_{p-1}^{\alpha-2}\\
 - (\alpha - 4)\mbinom{\alpha - 1}{2} G_{3(p - 1)+k_0}E_{p-1}^{\alpha-3}        
 +\mbinom{\alpha - 1}{3}  G_{4(p - 1)+k_0} E_{p-1}^{\alpha-4}
\pmod{p^4}.
\end{multline}

If $\alpha\equiv 1\pmod p$ then there are $\lambda_1,\lambda_2,\lambda_3,\lambda_4\in\mathbb{Z}_{(p)}$ such that
\begin{multline}
    G_{\alpha(p - 1)+k_0 }\equiv 
    \lambda_1 G_{(p-1)+k_0}E_{p-1}^{\alpha-1}
    +p\lambda_2G_{k_0+2(p-1)}E_{p-1}^{\alpha-2}\\
    +p\lambda_3G_{3(p-1)+k_0}
    +p\lambda_4G_{4(p-1)+k_0}E_{p-1}^{\alpha-4}
    \pmod{p^4}.
\end{multline}
We then use \eqref{eq:fact_filt_fact1} and 
Corollary~\ref{cor:main_factor_filt_p_cubed} to conclude that 
\begin{gather}
 \tilde\omega_{p^4}\left(G_{\alpha(p-1)+k_0}\right)\leq \begin{cases}
     2(p-1)+k_0,&\text{if }k_0=4;\\
     3(p-1)+k_0,&\text{if }k_0=2.
 \end{cases}   
\end{gather}
The remaining cases follow from similar analysis, and we omit the details.
\end{proof}

\section{Possible generalizations}\label{sec:speculation}
Computations suggest that the analogues of Theorem~\ref{thm:main_G_k_cong} and Corollary~\ref{cor:main_factor_filt} are true with $G_k$ replaced by $E_k$ in the case when $(p-1)\mid k$.
In other words, if $k^*>m$ is a multiple of $p-1$, then it appears that we have 
\begin{gather}\label{eq:E_k_conj}
    E_{\alpha(p-1)+k^*}
    \equiv \sum_{r=0}^{m-1}H(m, \alpha, r)E_{r(p-1)+k^*}E_{p-1}^{\alpha-r}\pmod{p^m}.
\end{gather}
From this it follows that for such $k$, 
with $k_0(m)$ as defined  in \eqref{eq:def_k0}, we have
\begin{gather}\label{eq:Ek_filt_speculation}
\tilde\omega_{p^m}\left(E_k\right)\leq (m-1)(p-1)+k_0(m).
\end{gather}

Note that if  $m<p$, then the results in  Theorem~\ref{thm:main_E_k_cong} and Corollary~\ref{cor:E_beta (p-1)_fact_filt} are stronger than the statements   \eqref{eq:E_k_conj} and \eqref{eq:Ek_filt_speculation}.
However, computations suggest that these statements are optimal for general $m$.

To prove these statements 
using the methods of this paper would require 
proving that if $k^*>m$ is a multiple of $p-1$  then for all 
$\alpha\geq 1$ we have
\begin{gather}\label{eq:ber_conj}
  \frac{\alpha(p-1)+k^* }{B_{\alpha(p-1)+k^*}}\equiv \sum_{r=0}^{m-1}H(m, \alpha, r)\,\frac {r(p-1) + k^*}{B_{r(p-1)+k^*}}\pmod{p^m}.
\end{gather}
We have verified the truth of \eqref{eq:ber_conj} when 
$5\leq p<100$, $p \le m \le 2p$, $m \le \alpha \le m+p$, and   $k^*$ is the smallest multiple  of $p-1$ larger than $m$.

\bibliographystyle{plain}
\bibliography{ref.bib}

\end{document}